\documentclass{amsart}
\usepackage{amsmath, amssymb, amsthm, color}
\usepackage[utf8x]{inputenc}
\usepackage{dsfont}
\usepackage{hyperref}
\usepackage[all]{xy}

\DeclareMathOperator{\Ima}{Im}

\newtheorem{theorem}{Theorem}[section]
\newtheorem{lemma}[theorem]{Lemma}
\newtheorem{proposition}[theorem]{Proposition}

\newtheorem{corollary}[theorem]{Corollary}
\newtheorem{prevtheorem}{Theorem}

\theoremstyle{definition}
\newtheorem{definition}[theorem]{Definition}

\theoremstyle{remark}
\newtheorem{remark}[theorem]{Remark}

\title{On the symmetry of images of word maps in groups}
\author{William Cocke \and Meng-Che Ho}
\date{\today}

\begin{document}

\begin{abstract}
Word maps in a group, an analogue of polynomials in groups, are defined by substitution of formal words. In \cite{Lubotzky}, Lubotzky gave a characterization of the images of word maps in finite simple groups, and a consequence of his characterization is the existence of a group $G$ such that the image of some word map on $G$ is not closed under inversion. We explore sufficient conditions on a group that ensure that the image of all word maps on $G$ are closed under inversion. We then show that there are only two groups with order less than $108$ with the property that there is a word map with image not closed under inversion. We also study this behavior in nilpotent groups.

\end{abstract}

\maketitle

\section{Introduction}
In this paper we consider subsets of a group that arise as the realizations of a word. The most famous example of such a subset is the set of commutators $[x,y]$, the realizations of the word $x^{-1} y^{-1} x y$, within a group. It is a basic property of commutators that $[x,y]^{-1}=[y,x]$, and hence the set of commutators in any group $G$ is closed under inverses in $G$. We investigate groups and words whose corresponding sets are not closed under inversion. 

In Dan Segal's \emph{Words} \cite{Words} a word $w$ is an expression of the form 
\[
w(x_1,\dots,x_k)=\prod_{j=1}^{s} x_{i_j}^{e_j},
\] where $i_1,\dots,i_j \in \{1,\dots,k\}$ and each $e_j$ is $\pm1$. For any group $G$, let $G^{(k)}$ be the direct product of $k$ copies of $G$, then we have the verbal mapping:
\[
w:G^{(k)}\rightarrow G, 
(g_1,\dots,g_k)\mapsto w(g_1,\dots,g_k)=\prod_{j=1}^{s} g_{i_j}^{e_j}.
\] 

We can also think of $w$ as an element of the free group on the symbols $(x_1,\dots,x_k)$. The map $w:G^{(k)}\rightarrow G$ is evaluation of $w$ using a $k$-tuple of elements of $G$. We write $G_w$ for the image of the map $w$ in $G$. Note, some authors including Segal \cite{Words}, set $G_w=w(G^{(k)})^{\pm 1}$ where $w$ is a word in $k$ variables. For an integer $s$, we define $G_w^s$ to be the set of elements of $G_w$ taken to the $s$-power. We will use the analogous notation $G^s$. We \textbf{do not} assume that $G_w$ is closed under inverses and instead we investigate the following property:

\begin{definition}
A pair $(G,w)$, where $G$ is a group and $w$ is a word, is called \emph{chiral} if $G_w\neq G_{w^{-1}}$. Equivalently, the pair $(G,w)$ is chiral if the set $G_w^{-1}$, the inverses of elements of $G_w$, does not equal $G_w$. We say $G$ is \emph{chiral} if $(G,w)$ is chiral for some $w$. Otherwise $G$ is \emph{achiral}. We say $x\in G$ witnesses the chirality of $G$ if $x\in G_w$ and $x^{-1} \notin G_w$ for some $w$.
\end{definition}

The existence of chiral groups can be shown from a result of Lubotzky \cite{Lubotzky}: In a finite simple group $G$ the images of word maps are exactly the subsets of $G$ closed under automorphisms and containing the identity. Consider $G=M_{11}$, the Mathieu group of order 7920. $G$ is chiral since an element of order 11 is not conjugate to its inverse and $\text{Out}(G)$ is trivial. However, we do not currently know of any word $w$ such that $(G,w)$ is chiral. 

In this paper, we begin the process of classifying all finite chiral groups. We prove: 
\begin{prevtheorem} \label{Listing}
The only chiral groups with order less than 108 are SmallGroups $(63,1)$ and $(80,3)$. \end{prevtheorem}

Section \ref{prop-of-chi} demonstrates how structural information about a group can force achirality of the group; for example finite Frobenius groups with abelian kernel and achiral complement are achiral. 

In section \ref{comp-of-chi} we recall an algorithm of Neumann \cite{Neumann}, for constructing all word maps on a finite group with a given number of variables. We prove the following theorem:

\begin{prevtheorem} \label{D-variables}
If a group $G$ is generated by $d$ elements, then $G$ is chiral if and only if there a word $w$ on $d$ variables such that $(G,w)$ is chiral. 
\end{prevtheorem} 

Hence the chirality of a finite group is recursive. In section \ref{inf-family-of-chi} we give an explicit infinite family of pairs of finite groups and words that have no nontrivial chiral quotient groups. Interestingly both chiral groups in Theorem \ref{Listing} are in this infinite family. In section \ref{nilp} we turn our attention to nilpotent groups. We show:

\begin{prevtheorem}
The free nilpotent groups of class $\geq 3$ are chiral and the free nilpotent groups of rank 3 and class 2 are achiral.
\end{prevtheorem}

\section{Properties of Chirality} \label{prop-of-chi}

Most of our results about achirality come from the following line of reasoning: a group $G$ is chiral if and only if there is a word $w$ such that $(G,w)$ is chiral. $(G,w)$ is chiral if and only if some element $x\in G$ that witnesses the chirality of $(G,w)$. If no element $x$ can be a witness of the chirality of $(G,w)$ for any $w$, then $G$ is achiral.  

For a group $G$, and $x,y\in G$ we say $x$ is \emph{automorphic} to $y$ if there is an automorphism $\sigma$ of $G$ such that $\sigma(x)=y$. We likewise will say that $x$ is homomorphic to $y$ if there is a homomorphism $\phi$ from $G$ to $G$ such that $\phi(x)=y$. Clearly, an element $x\in G$ cannot be a witness to chirality if $x$ is homomorphic to $x^{-1}$. This gives us the following simple observation:
\begin{lemma}\label{autmorphic to inverse}
Let $G$ be a group with the property that for every $x\in G$ there is a $\phi$, a homomorphism of $G$ (dependent on $x$), such that $\phi(x)=x^{-1}$. Then $G$ is achiral. 
\end{lemma} 

The following lemmas will be useful in developing a computational test for chirality. The first lemma shows that $G_w$ depends only one equivalence class of $w$ via automorphisms in a free group containing $w$. This lemma was first observed via examples using the automated proof software Prover9 \cite{Prover9}. Using Prover9 the authors showed that the images of maps associated with the 2-Engel word $[x,y,y]$ and 3-Engel word $[x,y,y,y]$ are closed under inversion in all groups. For $n$ greater than 3, we do not know whether the image of the $n$-Engel word is closed under inversion for all groups. 

\begin{lemma}
Let $F_n$ be the free group on the symbols $x_1,\dots,x_n$ and let $w\in F_n$ be a word. Let $\sigma \in \text{Aut}(F_n)$ and $u=\sigma(w)$. Then for any group $G$, $G_w = G_u$. 
\end{lemma}

\begin{proof}
The elementary Nielsen transformations do not change the image of a word map in a group.  
\end{proof}

We will use Nielsen transformations to show that for a group $G$ and a word $w$ $(G,w)$ is chiral if and only if $(G,v)$ is chiral for a word $v$ of specified form:

\begin{lemma}
Let $F_n$ be the free group on the symbols $x_1,\dots,x_n$ and let $w\in F_n$ be a word. There is an automorphism of $F_n$ taking $w$ to a word $v$ of the form $x_1^{a} c$, where $c$ is in $F_n'$. 
\end{lemma}

\begin{proof}
Let $w(x_1,x_2,\ldots,x_n)\in F_n$ be a word, where $F_n$ is the free group on the symbols $x_1,\dots,\,x_n$. We say the \emph{weight} of $w$ is the tuple $(a_1,\ldots,a_n)$, where $a_i$ is the sum of powers of $x_i$ in $w$. Since $w(x_1,x_2,\ldots,x_n)$ and $w(x_1x_i,x_2,\ldots,x_n)$ have the same image in any group,  we may run Euclid's algorithm on the weight tuple. By iteratively pushing letters to the front of the word, we may further assume that $w$ has the form $x_1^{a}c$, where $c$ is in $F_n'$.
\end{proof}

Now let $G$ be a group with exponent $e$ and $w=x^a c$ where $c$ is in $F_n'$ for some $n$. Let $d=\gcd(e,a) = re+sa$. Since $\gcd(s,e) = 1$, we have $G^s = G$, and thus $w(x_1,x_2,\ldots,x_n)$ and $w(x_1^s,x_2^s,\ldots,x_n^s)$ have the same image, i.e.,~$\Ima(x_1^ac) = \Ima(x_1^{sa}c') = \Ima(x_1^{d-re}c') = \Ima(x_1^dc')$. We have shown that a group $G$ is chiral if and only if it chiral for words of a specific form:

\begin{theorem}
A group $G$ with exponent $e$ is chiral if and only if $(G,w)$ is chiral for some word $w = x^kc$ where $k$ divides $e$ and $c$ is a product of commutators.
\end{theorem}

We will use the above proposition to prove the following:

\begin{corollary}\label{G^kG'}
Let $G$ be a group with finite exponent $e$. Let $G^k$ be the set of $k$th powers in $G$. If for every $k$ dividing $e$, every element of $G^k G' \setminus G^k$ is not a witness of chirality, then $G$ is achiral. 
\end{corollary}

\begin{proof}
Consider a word $w$ of the form $x^k c$. Clearly, the image of $w$ is inside $G^k G'$. Suppose by way of contradiction that $(G,w)$ is chiral as witnessed by $g\in G$. Then $g\in G^k G'$ and by hypothesis $g\in G^k$. Let $g=h^k$ for some $h\in G$. Hence $w(h,1,\dots,1)=g$. But, then $g^{-1}=h^{-k}$ and $w(h^{-1},1,\dots,1)=g^{-1}$, contradicting $g$ as a witness to the chirality of $(G,w)$. Therefore $G$ is achiral. 
\end{proof}

The next few results show how the structure of a group $G$ limits the potential witnesses to the chirality of $G$. 

\begin{lemma}\label{inversion}
Let $N$ be an abelian group. Suppose another group $H$ acts on $N$ via automorphisms and consider $G=N\rtimes H$. Then there is an automorphism of $G$ that acts as inversion on $N$ and fixes $H$.
\end{lemma}

\begin{proof}
Let $\sigma$ be a set map that inverts $N$ and fixes $H$ point-wise. We will show that $\sigma$ is a homomorphism of $G$. We will denote the action of $h$ on $n$ via $n^h$. Let $n,m\in N$ and $h,j\in H$ then 
\begin{flalign*}
(nh mj)^{\sigma} &= (nm^{h^{-1}} h j)^{\sigma} \\
&=(m^{h^{-1}})^{-1} n^{-1} h j \\
&=n^{-1}h m^{-1}j \\
&=(nh)^{\sigma} (mj)^{\sigma}.
\end{flalign*}

Hence $\sigma$ is a homomorphism. It is clearly surjective and injective, thus an automorphism. 
\end{proof}

\begin{lemma}\label{semi-direct component}
Let $G$ be a group with a normal subgroup $N$, such that $N$ is complemented in $G$ by $H$. Let $w$ be a word on $d$ variables. Suppose $w(g_1,\dots,g_d)=h\in H$ for some $g_1,\dots,g_d\in G$. Write $g_i=n_i h_i$. Then $w(h_1,\dots,h_d)=h$.
\end{lemma}

%THIS IS A VERY NICE PROOF THAT AVOIDS NASTY COORDINATES

\begin{proof}
By considering the action of $H$ on $N$, $w(g_1,\dots,g_d)$ can be written as $n\cdot w(h_1,\dots,h_d)$ for some $n\in N$. Since $N\cap H=1$ we conclude that $n=1$ and $w(h_1,\dots,h_d)=h$. 
\end{proof}

\begin{corollary}\label{semi-direct witnesses}
Let $N$ and $H$ be groups and $G=N\rtimes H$. If $H$ is achiral, then no element of $H$ can witness the chirality of $G$. 
\end{corollary}

\begin{proof}
For a group $G=N\rtimes H$, %with $N\cap H =1$, 
$h\in H$ is a witness to the chirality of $(G,w)$ if and only if $h\in H$ is a witness to the chirality of $(H,w)$. 
\end{proof}

\begin{theorem}\label{abelian normal subgroup splits}
Let $N$ be an abelian group and $H$ be an achiral group with $G = N \rtimes H$. If every element of $G$ is automorphic to an element of either $N$ or $H$, then $G$ is achiral.
%If a group $G$ has an abelian normal subgroup $N$ with achiral complement $H$ in $G$, such that every element of $G$ is automorphic to an element of $N$ or $H$ then $G$ is achiral.
\end{theorem}

\begin{proof}
We will show that no element of $G$ can witness the chirality of $(G,w)$. Since the images of word maps are closed under automorphisms, if $(G,w)$ was chiral, there would be a witness in either $N$ or $H$.

From Lemma \ref{inversion} there is an automorphism $\sigma$ of $G$ that acts as inversion on $N$ and fixes $H$. Since the image of a word map is closed under automorphisms, no element of $N$ can witness the chirality of $(G,w).$ Corollary \ref{semi-direct witnesses} states that no element of $H$ can witness the chirality of $(G,w)$. Hence $G$ is achiral. 
\end{proof}

We can restate Theorem \ref{abelian normal subgroup splits} in terms of witnesses:
\begin{theorem}
Let $N$ be an abelian group and $H$ be an achiral group with $G = N \rtimes H$. No witness of the chirality of $G$ can be automorphic to an element of $N$ or $H$. 
\end{theorem}

\section{The Computability of Chirality} \label{comp-of-chi}

We will give an algorithm that determines if a given finite group is chiral. As part of our algorithm, we will calculate for all words on some specified number of variables all of the sets $G_w$; this part of our algorithm is similar to an algorithm originally discovered by Neumann in \cite{Neumann}. We start by proving the following lemma, which says that the chirality of a finitely generated group is only dependent on words of a given number of variables. 

\begin{lemma}\label{locality of chirality}
If a group $G$ is generated by $d$ elements, then $G$ is chiral if and only if there is a word $w$ on $d$ variables such that $(G,w)$ is chiral.
\end{lemma}

\begin{proof}
We need only to show that if $G$ is chiral, then there is a word $w$ on $d$ variables that witnesses the chirality of $G$. 

Fix a generating set $g_1,g_2,\ldots,\,g_d$ of $G$, and fix $d$-variables words $u_1,u_2,\ldots,\,u_{|G|}$, such that $u_i(\overline{g})$ enumerates $G$. Let $v$ be a word with $k$ variables. Then \[v(u_{i_1},u_{i_2},\ldots,u_{i_k})(G^d) \subseteq v(G^k).\] On the other hand, every $k$-tuple from $G$ can be written as $(u_{i_1}(\overline{g}),\ldots,u_{i_k}(\overline{g}))$, so we have
$$\bigcup\limits_{1\leq i_j\leq |G|} v(u_{i_1},u_{i_2},\ldots,u_{i_k})(G^d) = v(G^k)$$

Now, if $(G,v)$ is chiral, then $v(G^k)$ is not closed under inverse for some $v$, thus there are some $\overline{i}$ such that $w = v(u_{i_1},u_{i_2},\ldots,u_{i_k})$ witnesses the chirality of $G$.
\end{proof}

We see that to check chirality, it suffices to check the $d$-variable words. Indeed, these word maps naturally form a group.

\begin{definition}
For a $d$-generated group $G$, define $W(G)$ to be the set of all word maps on $d$ variables on $G$. For $w,u \in W(G)$, we define $w\cdot u$ to be the word map given by $w$ concatenated with $u$, which is the same as point-wise multiplication as maps. The gives $W(G)$ a group structure, and it naturally embeds into the direct product $G^d$.
\end{definition}

We are now ready to state Neumann's algorithm:

\begin{theorem}
There is an algorithm, when given a finite group as input, outputs whether $G$ is chiral.
\end{theorem}

\begin{proof}
Let $G$ be $d$-generated. We first build the Cayley graph of $W(G)$. For a vertex with label $w$, a word in $F_d$, there is an outward edge labeled $x_i$ that connects to a vertex labeled by the word map $wx_i$. We then check if this (as a map) is equal to some existing vertex. For every existing vertex, the check is finite since the group is finite, and there are only finitely many existing vertices. This process terminates since $W(G)$, being a subgroup of $G^d$, is finite. Now for each vertex, we check if the (finite) image of the map is chiral. If it is chiral for any word map, we return \emph{chiral}, otherwise we return \emph{achiral}.
\end{proof}

This construction is actually related to the theory of varieties of groups. We will digress, and elaborate more on this relation. Most, if not all, of these are rephrasing results from \cite{Neumann}. We first recall some terminologies from \cite{}[MKS?].

\begin{definition}
A group $G$ is said to satisfy a \emph{law} $w\in F_d$, if $w(G) = 1$.

A \emph{variety of groups} are the class of groups satisfying a given collections of laws.

A subgroup $H$ of $G$ is called a \emph{verbal subgroup} if it is generated by the image of some word map.

A group $G$ is called a \emph{reduced free group} if it is the quotient of a free group by a verbal subgroup. These are also the ``free objects" (among the $d$-generated groups) in some variety of groups.

For a $d$-generated group $G$, we write $FV(G)$ to be the reduced free group of rank $d$ in the variety generated by $G$.
\end{definition}

For example, free groups, free abelian groups, and free nilpotent groups are all examples of reduced free groups. 

Note that the projection maps $\pi_i:G^d \to G$ are realized as word maps $x_i$, and they form the standard generating set of $W(G)$. This gives an identification of the groups $W(G)$ and $FV(G)$.

\begin{proposition}
Let $G$ be a $d$-generated group. Then $W(G)\cong FV(G)$.
\end{proposition}

\begin{proof}
Note that $FV(G)$ has the standard presentation $\langle a_1,\ldots a_d \mid R\rangle$, where $R$ is the set of all the laws $G$ satisfies, substituted by words in $a_i$'s. Consider the map $\phi$, sending $x_i$ to $a_i$. This can be extended to an automorphism. Indeed, $w = 1$ in $W(G)$ if and only if $G$ satisfies $w$ as a law, which happens exactly when $\phi(w) = 1$. Thus $W(G) \cong FV(G)$.
\end{proof}

Note that this also shows that for a finite group $G$, $FV(G)$ is finite, since $W(G) \subseteq G^d$ is finite. Furthermore, the algorithm actually builds the group $FV(G)$, and in particular, enumerates all the $d$-variable laws satisfied by $G$. In \cite{Neumann}, it was pointed out that this can be used to find all laws satisfied by $G$ with a bounded number of variables, but does not give a finite process to find \emph{all} laws. In our case, since chirality can be reduced to a property on words with a bounded number of variables, it suffices to stop at a finite stage, hence yielding an algorithm.

In practice, Neumann's algorithm and our implementation of it are time and memory intensive and do not yield a practical method for determining if a finite group is chiral. For example, $FV(S_3)$ has order 972. A result of Waldemar Ho\l{}obowski \cite{Holobowski} shows that if $G$ is SmallGroup(20,3) then \[|FV(G)| = 122070317250000.\] 

\begin{remark}
We abuse notation and say a word $w$ is chiral if there is some group for which $(G,w)$ is chiral. The chirality of a word is decidable, i.e., there is an algorithm, when input a word, outputs whether the word is chiral or not. Indeed, Given a word $w$, the word is chiral if and only if it is chiral in some free group, which is equivalent to saying the free group does not satisfy the first-order sentence $\forall \overline{x} \exists \overline{y} w(\overline{x}) \cdot w(\overline{y}) = 1$. This is a sentence in the positive theory of the free group, which coincides for all nonabelian free groups \cite{Merzljakov} and is decidable \cite{Makanin}. 
\end{remark}

\section{An Infinite Family of Chiral Groups}\label{inf-family-of-chi}
We first note that achirality is preserved under quotienting:

\begin{proposition}
Let $H$ be a homomorphic image of $G$. If $H$ is chiral, then $G$ is also chiral.
\end{proposition}

\begin{proof}
Let $\phi$ be the homomorphism taking $G$ onto $H$. Let $w$ be a word witnessing the chirality of $H$, i.e.~there is $h\in w(H)$ with $h^{-1} \notin w(H)$. Suppose $w(\overline{y}) = h$ for $\overline{y}\in H$. Let $\overline{x}\in G$ be such that $\phi(\overline{x}) = \overline{y}$, and we have $\phi(w(\overline{x})) = h$. Write $g = w(\overline{x})$.

Now we claim that $g^{-1} \notin w(G)$. If not, let $w(\overline{x}') = g^{-1}$. Then $h^{-1} = \phi(g^{-1}) = \phi(w(\overline{x}')) = w(\phi(\overline{x}')) \in w(H)$, a contradiction. Thus $G$ is also achiral as witnessed by $w$ and $g$.
\end{proof}

Hence a classification of all finite chiral groups depends only on classifying those that do not have a proper chiral quotient. We call such groups \emph{minimal chiral}. The next theorem shows the existence of an infinite family of minimal chiral groups. Moreover, both SmallGroup(63,1) and SmallGroup(80,3) are part of this family. 

\begin{lemma}\label{family}
Let $C_p$ acts on $C_q$ by multiplying by $\phi$, and assume that $\phi$ has order $p$ and $\phi-1$, $\phi+1$ are both coprime to $q$. Then for $p \mid r$, $C_q \rtimes C_{pr}$ with the action multiplying by $\phi$ is chiral with the word $w=a^p[a,b][a^{-1},b]^\phi$.
\end{lemma}

\begin{proof}
Write $a = (x,n)$ and $b = (y,m)$ in $C_{pr}\ltimes C_q$. We compute:
$$a^p = (px,n+\phi^x n+\cdots+\phi^{x(p-1)}n)$$
$$[a,b] = (0,-n-\phi^x m+\phi^y n + m)$$
$$[a^{-1},b] = (0,\phi^{-x}n-\phi^{-x}m-\phi^{y-x}n+m)$$

Note that when $x = 1$, $a^p = (p,0)$ since $\gcd(\phi-1,q)=1$ implies $1+\phi+\cdots+\phi^{p-1}\equiv 0 (q)$.

Note that the first coordinate of $w$ is $p$ if and only if $px = p (pr)$, thus $x = 1 (r)$. Note that $\phi^r = 1 (q)$ since the action has order $p\mid r$. Thus, when $x=1(r)$, we have:
$$a^p = (p,0)$$
$$[a,b]=(0,-n-\phi m+\phi^y n+m)$$
$$[a^{-1},b]=(0,\phi^{-1}n-\phi^{-1}m-\phi^{y-1}n+m)$$

Thus, $a^p[a,b][a^{-1},b]^\phi = (p,0)$ if the first coordinate is $p$.

Now consider when the first coordinate is $-p$. Again, this happens if and only if $x=-1(r)$. We again compute:
$$a^p = (-p,0)$$
$$[a,b]=(0,-n-\phi^{-1} m+\phi^y n+m)$$
$$[a^{-1},b]=(0,\phi n-\phi m-\phi^{y+1}n+m)$$
And
\begin{align*}
a^p[a,b][a^{-1},b]^\phi &= (-p,-n-\phi^{-1} m+\phi^y n+m+\phi^2 n-\phi^2 m-\phi^{y+2}n+\phi m)\\
{}&=(-p,(\phi+1)(\phi-1)(1-\phi^y)n + (1+\phi)(1-\phi)(1-\phi^{-1})m)
\end{align*}

Since $\phi-1$ and $\phi+1$ are both coprime to $q$, as $n,m,y$ ranges over various values, this ranges over the coset $(-p,0)C_q$. The inverse of the coset $(-p,0)C_q$ is $(p,0)C_q$, but the image of the word does not include any elements of the form $(p,0)C_q$ except for $(p,0)$. Therefore, $(G,w)$ is chiral.

\end{proof}

Hence we have shown that SmallGroup$(63,1)$ and SmallGroup$(80,3)$ are achiral. We can now prove Theorem \ref{Listing}, restated below:
\begin{theorem}
The only chiral groups with order less than 108 are SmallGroup $(63,1)$ and SmallGroup $(80,3)$.
\end{theorem}

\begin{proof}
Recall that any chiral group must have an element $x\in G$ such that $\sigma(x)\neq x^{-1}$ for all automorphisms $\sigma$ of $G$. There are only 44 groups with this property of order less than 108. Of those 44 groups, only SmallGroup $(63,1)$, SmallGroup$(80,3)$, and SmallGroup$(81,10)$ are not shown to be achiral by Corollary \ref{G^kG'} as tested using Magma \cite{Magma}. From above we know that SmallGroup$(63,1)$ and smallGroup$(80,3)$ are chiral. 

To show that SmallGroup$(81,10)$ is achiral, we performed a search in Magma over the Mal'cev coordinates of the free nilpotent group of exponent 9 and class 3 to show that the free nilpotent group on two generators of exponent 9 and class 3 is achiral. Since SmallGroup$(81,10)$ is a quotient of the free nilpotnet group on two generators of exponent 9 and class 3, we conclude that SmallGroup$(81,10)$ is achiral.  
\end{proof}

\section{Nilpotent Groups}\label{nilp}

It is clear that every abelian group is achiral. However, in this section we will see that there are chiral nilpotent groups.

\begin{lemma}\label{homomorphic lemma}
A reduced free group $G$ is achiral if and only if every element is homomorphic to its inverse.
\end{lemma}

\begin{proof}
Suppose first every element in $G$ is homomorphic to its inverse. Then if $g \in \Ima(w)$ for some word map $w$ and $g\in G$, we have $w(\overline{a}) = g$ for some $\overline{a}\in G$ and the homomorphism $\phi$ sending $g$ to $g^{-1}$ gives $w(\phi(\overline{a})) = g^{-1}$, so $g^{-1}\in \Ima(w)$. Thus $G$ is achiral.

Now suppose $G$ is achiral and let $g \in G$. Fix a generating set $S$ of $G$ and write $g$ as a word $w$ in $S$. Considering $w$ as a word map, we see $g \in \Ima(w)$ by evaluating on $S$. By achirality of $G$, we have $g^{-1}\in \Ima(w)$, say by evaluating on $T$. Consider the map that maps elements of $S$ to corresponding elements of $T$. Since $G$ is a reduced free group, this map can be extended to an homomorphism on $G$, and it maps $g$ to $g^{-1}$.
\end{proof}

\begin{theorem}
The class 2 rank 3 free nilpotent group $N = N_{2,3}$ is achiral. As a result, every class 2 rank 3 nilpotent group is achiral.
\end{theorem}

\begin{proof}
Write the generators of $N$ to be $a,b,c$ and the commutators to be $d=[a,b]$, $e=[a,c]$, $f=[b,c]$. Since every element in $N$ is automorphic to some element of the form $a^*d^*e^*f^*$, it suffices to show that elements of this form is homomorphic to its inverse.

Fix $g = a^id^je^kf^l \in N$. Consider the homomorphism $\phi$ with $\phi(a) = a^{-1}$, $\phi(b) = b^xc^y$, and $\phi(c) = b^zc^w$. We have $\phi(d) = d^{-x}e^{-y}$, $\phi(e) = d^{-z}e^{-w}$, and $\phi(f) = f^{xw-zy}$. Thus, to have $\phi(g) = g^{-1}$, we need $$a^{-i}d^{-j}e^{-k}f^{-l} = (a^{-1})^i(d^{-x}e^{-y})^j(d^{-z}e^{-w})^k(f^{xw-zy})^l,$$ which is equivalent to the following system of equations:
$$\begin{cases}
xj+zk = j\\
yj+wk = k\\
xw-zy = -1.
\end{cases}$$
However, this is again equivalent to finding an integer matrix $M = \left(\begin{matrix}
x & z\\
y & w
\end{matrix}\right)$ such that its determinant is -1 and the vector $\left(\begin{matrix}j\\k\end{matrix}\right)$ is its eigenvector with eigenvalue 1. This matrix can be found by starting with the matrix $\left(\begin{matrix}1&0\\0&-1\end{matrix}\right)$ and do a change of bases such that $(\gcd(j,k),0)$ gets mapped to $(j,k)$. Thus, we see that every $g\in N$ is homomorphic to its inverse, and thus $N$ is achiral by the previous lemma.

\end{proof}

\begin{theorem}
Let $N_{3,2} = \langle a,b \rangle$ be the free nilpotent group of class 3, rank 2, and let $c = [a,b]$, $d = [c,a]$, and $e = [c,b]$ be the standard Mal'cev basis of $N_{3,2}$. Then for any odd prime $p$, the element $g = a^{p^2}c^pd$ is not homomorphic to its inverse. Thus, $N_{3,2}$ is chiral.
\end{theorem}

\begin{proof}
Suppose $\phi$ is an automorphism such that $\phi(g) = g^{-1} = a^{-p^2}c^{-p}d^{p^3-1}$. For simplicity, we will use $*$ to denote unknown (possibly different for different $*$'s) integers, and $n*$ to denote integers divisible by $n$. By considering the power of $a$ in $\phi(g)$, we see $\phi(a)$ must have the form $a^{-1}c^*d^*e^*$. Suppose $\phi(b) = a^*b^xc^*d^*e^*$. Thus, $\phi(c) = c^{-x}d^*e^*$ and $\phi(d) = d^{x}$.

We then compute 
\begin{align*}
\phi(g) &= (a^{-1}c^*d^*e^*)^{p^2}(c^{-x}d^*e^*)^p(d^{x}) \\
 &= (a^{-p^2}c^{p^2*}d^{-\frac{p^2(p^2-1)}{2}*+p^2*}e^{p^2*})(c^{-px}d^{p*}e^{p*})(d^x)\\
 &= (a^{-p^2}c^{p^2*-px}d^{p*+x}e^{p*})
\end{align*}

By considering the exponent of $c$ modulo $p^2$, we see $-px \equiv -p$ modulo $p^2$, so $x \equiv 1$ modulo $p$. However, considering the exponent of $d$ modulo $p$, we get $x \equiv -1$ modulo $p$, a contradiction. Thus the theorem follows.
\end{proof}

\begin{corollary}
If $G$ is a finite nilpotent group of class 3 and rank 2, then $G$ is chiral.
\end{corollary}

\begin{remark}
The previous argument and hence chirality still holds for (finite) quotients of the free nilpotent group with the order of $a$ being infinity or divisible by $p^3$, order of $b$ being infinity or divisible by $p^2$, and order of $c$ being infinity or divisible by $p$.
\end{remark}

\bibliographystyle{amsalpha}
\bibliography{sample}

\end{document}